\documentclass{amsart}
\usepackage{amssymb, amsmath, amsthm, amscd}
\usepackage[all, cmtip]{xy}
\usepackage{fullpage, graphicx}
\usepackage[T1]{fontenc}
\usepackage{etoolbox}

\newcommand{\Cbb}{{\mathbb{C}}}

\newcommand{\Nbb}{{\mathbb{N}}}

\newcommand{\Pbb}{{\mathbb{P}}}
\newcommand{\Qbb}{{\mathbb{Q}}}

\newcommand{\Zbb}{{\mathbb{Z}}}

\newcommand{\Nb}{{\mathbf{N}}}

\newcommand{\Ocal}{{\mathcal{O}}}

\newcommand{\afrak}{{\mathfrak{a}}}

\newcommand{\ffrak}{{\mathfrak{f}}}

\newcommand{\mfrak}{{\mathfrak{m}}}

\newcommand{\pfrak}{{\mathfrak{p}}}

\newcommand{\id}{{\textup{id}}}
\newcommand{\Het}{\operatorname{H}_{\acute{\text{e}}\text{t}}}

\DeclareMathOperator{\Aut}{Aut}

\DeclareMathOperator{\Frob}{Frob}
\DeclareMathOperator{\Gal}{Gal}
\DeclareMathOperator{\GL}{GL}

\DeclareMathOperator{\Img}{Im}

\DeclareMathOperator{\Real}{Re}

\DeclareMathOperator{\tr}{tr}

\theoremstyle{definition}
		\newtheorem{definition}{Definition}

\theoremstyle{plain}
		\newtheorem{theorem}[definition]{Theorem}

		\newtheorem{lemma}[definition]{Lemma}

		\newtheorem{conjecture}[definition]{Conjecture}
\theoremstyle{remark}
		\newtheorem*{remark}{Remark}

\apptocmd{\thebibliography}{\fontsize{10}{15}\selectfont}{}{}

\begin{document}

\title{Arithmetic and intermediate Jacobians of some rigid Calabi-Yau threefolds}
\author{Alexander Molnar}
\address{Department of Mathematics and Statistics\\
  Queen's University\\
  Kingston, Ontario, K7L 3N6}
\email{a.molnar@queensu.ca}
\date{\today}
\begin{abstract}
We construct Calabi-Yau threefolds defined over $\Qbb$ via quotients of abelian threefolds, and re-verify the rigid Calabi-Yau threefolds in this construction are modular by computing their L-series, without \cite{Dieulefait} or \cite{GouveaYui}. We compute the intermediate Jacobians of the rigid Calabi-Yau threefolds as complex tori, then compute a $\Qbb$-model for the 1-torus given a $\Qbb$-structure on the rigid Calabi-Yau threefolds, and find infinitely many examples and counterexamples for a conjecture of Yui about the relation between the $L$-series of the rigid Calabi-Yau threefolds and the $L$-series of their intermediate Jacobians. 
\end{abstract}

\maketitle

\section{Introduction}

Associated to any smooth projective variety, the intermediate Jacobian varieties generalize Jacobian varieties of curves. For $X$ an $n$-dimensional complex variety we define the (Griffiths) intermediate Jacobians of $X$ to be the varieties
\[ J^{k+1}(X) := H^{2k+1}(X,\Cbb)/(F^{k+1}H^{2k+1}(X,\Cbb)\oplus H^{2k+1}(X,\Zbb)) \]
where the quotient involves only the torsion free part of $H^{2k+1}(X,\Zbb)$ and $F^{k+1}$ is the $(k+1)$-th level in the Hodge filtration, i.e.,
\[ F^{k+1}H^{2k+1}(X,\Cbb) = \bigoplus_{\substack{p+q=2k+1 \\ q<k+1}} H^{p,q}(X). \]
Thus, for a curve $C$, the only intermediate Jacobian is $J^1(C)$, which is the Jacobian variety of the curve. We will be interested mostly in Calabi-Yau varieties, so for the one-dimensional case we have elliptic curves which are known to be isomorphic to their (intermediate) Jacobians, and these are all the possibilities. For two dimensional examples one has K3 surfaces which have trivial first and third cohomology, hence no-non-trivial intermediate Jacobians, and so our focus will be with Calabi-Yau threefolds. Here the only non-trivial intermediate Jacobian is
\[ J(X) := J^2(X) = (H^{0,3}(X)\oplus H^{1,2}(X))/H^3(X,\Zbb). \]
Not much is known about these varieties, but they are very useful tools. As complex varieties one has e.g., \cite{ClemensGriffiths} who used the intermediate Jacobian of a cubic threefold to show there exist unirational varieties that are not rational, and \cite{VoisinIJ} where intermediate Jacobians are used to show the Griffiths group of a general member in a family of non-rigid Calabi-Yau threefolds is infinite dimensional. An interest in the physics literature comes from the use of Calabi-Yau varieties in string theory, e.g., in \cite{BKNPP}, \cite{DDP} and \cite{Morrison}. As $X$ is K\"ahler we have that the intermediate Jacobian is a complex torus of dimension $1+h^{1,2}(X)$. Thus, when $X$ is rigid, $J^2(X)$ is a 1-torus, possibly with a canonical structure of an elliptic curve.

The work to date studies the geometry of $J(X)$, i.e., the complex structure. However, as intermediate Jacobians generalize the classical Jacobian variety of a curve, as well as the Picard varieties and Albanese varieties of any $n$-dimensional varieties, it is natural to ask if one can study arithmetic on all intermediate Jacobians by finding a canonical $\Qbb$-structure when $X$ is defined over $\Qbb$. All of our examples of Calabi-Yau threefolds will be defined over $\Qbb$, but one can similarly try to associate a canonical $K$-structure for any number field $K$ if $X$ is defined over $K$ instead. We will construct some rigid Calabi-Yau threefolds in which we are able to compute their intermediate Jacobians as complex varieties, and then refine this computation to give a natural $\Qbb$-structure as well, given a choice of model for the Calabi-Yau threefolds defined over $\Qbb$.

Shafarevich conjectures that every variety of CM-type (meaning its Hodge group is abelian) has the $L$-series of a Grossencharacter, a Hecke $L$-series, and Borcea shows that a rigid Calabi-Yau threefold with CM-type has a CM elliptic curve as its intermediate Jacobian, which is well known to have a Hecke $L$-series. Thus, if the conjecture is true, it is a natural question to ask if there is any relation between the associated Grossencharacters of a rigid Calabi-Yau threefold and its intermediate Jacobian. The motivation for the current work is to study a precise conjecture of Yui to this effect.
\begin{conjecture}[Yui, \cite{YuiUpdate}] \label{Yui} Let $X$ be a rigid Calabi-Yau threefold of CM-type defined over a number field $F$. Then the intermediate Jacobian $J^2(X)$ is an elliptic curve with CM by an imaginary quadratic field $K$, and has a model defined over the number field $F$.

If $\chi$ is a Hecke character associated to $J^2(X)$ and
\[ L(J^2(X),s) = \left\{ \begin{array}{ll} L(\chi,s)L(\overline{\chi},s) & \text{if}\ K\subset F, \\
 L(\chi,s) & \text{otherwise,} \end{array}\right. \]
then
\[ L(X,s) = \left\{ \begin{array}{ll} L(\chi^3,s)L(\overline{\chi}^3,s) & \text{if}\ K\subset F, \\
 L(\chi^3,s) & \text{otherwise.} \end{array} \right. \]
Consequently, $X$ is modular. \end{conjecture}
We will show that many of our rigid Calabi-Yau threefolds of CM-type satisfy this conjecture, but not all. In particular, we will show that if one of our examples, $X$, satisfies the conjecture, then all quadratic twists of $X$ satisfy the conjecture, while all non-quadratic twists of $X$ do not. After this we generalize our construction to Calabi-Yau $n$-folds and show that for infinitely many $n$ satisfying a congruence with the order of the CM automorphisms we have the natural generalization of the conjecture is true. Similarly, for infinitely many $n$ the conjecture will not be true because of the CM twists on the varieties.

We start by constructing our varieties and studying their geometry over $\Cbb$. We use a generalized Borcea construction of Calabi-Yau threefolds using (finite) quotients of products of elliptic curves, and determine which quotients give rigid Calabi-Yau threefolds. We then choose a $\Qbb$-structure for the Calabi-Yau threefolds, via the underlying elliptic curves, and compute their respective $L$-series. Once this is done we compute the intermediate Jacobians as complex tori, and then over $\Qbb$, via the choice of $\Qbb$-structure given on the threefolds. We are then able to compare the $L$-series of the intermediate Jacobians and their respective threefolds and check the conjecture.

Our construction also gives rise to non-rigid Calabi-Yau threefolds, which we leave to future work \cite{Molnar3}, as the intermediate Jacobians are no longer elliptic curves, and the arithmetic of 2-tori and 4-tori is more complicated. Moreover, the question of modularity (automorphy) is nowhere near as resolved, as \cite{Dieulefait} and \cite{GouveaYui} no longer apply.

\subsection*{Acknowledgements} The author was partially supported by Ontario Graduate Scholarships, as well as partial support and hospitality at the Fields Institute for the thematic program on Calabi-Yau varieties in 2013, the University of Copenhagen in 2014, and the Leibniz Universit\"at Hannover in 2015. We are very grateful for these opportunities. We are also grateful for the patience and many insightful discussions we had throughout writing this work, in particular with Ian Kiming, Hector Pasten, Andrija Peruni\v{c}i\'c, Simon Rose, Matthias Sch\"utt, and our supervisor Noriko Yui, without whom we may never have been introduced to the interesting questions we address here.

\section{Construction of threefolds}

We will generalize a construction of Calabi-Yau threefolds due to Borcea \cite{Borcea}, using elliptic curves with complex multiplication. To start, we first consider our threefolds over $\Cbb$, determine their Hodge numbers.

Over $\Cbb$, there is only one elliptic curve with an automorphism of order 3 and one elliptic curve with an automorphism of order 4, up to isomorphism. Denote these by $E_3$ and $E_4$, with their respective CM automorphisms $\iota_3$ and $\iota_4$, and note that $\iota_6 := -\iota_3$ is an automorphism of order 6 on $E_6 := E_3$.

On the triple product $E_j^3 := E_j\times E_j\times E_j$ we have an action of the group $G_j := \langle \iota_j\times\iota_j^{j-1}\times\id,\iota_j\times\id\times\iota_j^{j-1}\rangle$ for $j=3,4$ and $6$. As $G_j$ preserves the holomorphic threeform of $E_j^3$ we have $h^{3,0}(E_j^3/G_j) = 1$ and $h^{1,0}(E_j^3/G_j) = 0$, so a crepant resolution of $E_j^3/G_j$ is a Calabi-Yau threefold. As $E_j^3/G_j$ is a global quotient orbifold of dimension 3, such a crepant resolution exists by \cite{BridgelandKingReid}. The same is true for many subgroups of $G_j$ but the geometry varies widely with the choice of subgroup.

\begin{theorem} \label{aut6} Consider the following groups of automorphisms acting on $E_6^3$.
\[ G_6 = \langle \iota_6\times\iota_6^5\times\id, \iota_6\times\id\times\iota_6^5\rangle, \qquad
H_6 = \langle \iota_6^2\times\iota_6^4\times\id, \iota_6^2\times\id\times\iota_6^4 \rangle, \]
\[ I_6 = \langle \iota_6^2\times\iota_6^4\times\id, \iota_6^4\times\iota_6\times\iota_6 \rangle, \qquad
J_6 = \langle \iota_6\times\iota_6^5\times\id, \iota_6^4\times\iota_6^5\times\iota_6^3 \rangle, \]
\[ K_6 = \langle \iota_6^3\times\iota_6^3\times\id, \iota_6^3\times\id\times\iota_6^3 \rangle, \qquad
L_6 = \langle \iota_6^3\times\iota_6^3\times\id, \iota_6^4\times\iota_6\times\iota_6 \rangle, \]
\[ M_6 = \langle \iota_6^2\times \iota_6^2\times \iota_6^2 \rangle, \qquad
N_6 = \langle \iota_6\times\iota_6^2\times\iota_6^3 \rangle, \qquad
O_6 = \langle \iota_6^4\times\iota_6\times\iota_6 \rangle. \]
Crepant resolutions of the respective quotients are Calabi-Yau threefolds with Hodge numbers
\begin{align*}
 h^{1,1}(\widetilde{E_6^3/G_6}) &= 84, & h^{2,1}(\widetilde{E_6^3/G_6}) &= 0, \\
 h^{1,1}(\widetilde{E_6^3/H_6}) &= 84, & h^{2,1}(\widetilde{E_6^3/H_6}) &= 0, \\
 h^{1,1}(\widetilde{E_6^3/I_6}) &= 73, & h^{2,1}(\widetilde{E_6^3/I_6}) &= 1, \\
 h^{1,1}(\widetilde{E_6^3/J_6}) &= 51, & h^{2,1}(\widetilde{E_6^3/J_6}) &= 3, \\
 h^{1,1}(\widetilde{E_6^3/K_6}) &= 51, & h^{2,1}(\widetilde{E_6^3/K_6}) &= 3, \\
 h^{1,1}(\widetilde{E_6^3/L_6}) &= 36, & h^{2,1}(\widetilde{E_6^3/L_6}) &= 0, \\
 h^{1,1}(\widetilde{E_6^3/M_6}) &= 36, & h^{2,1}(\widetilde{E_6^3/M_6}) &= 0, \\
 h^{1,1}(\widetilde{E_6^3/N_6}) &= 35, & h^{2,1}(\widetilde{E_6^3/N_6}) &= 11, \\
 h^{1,1}(\widetilde{E_6^3/O_6}) &= 29, & h^{2,1}(\widetilde{E_6^3/O_6}) &= 5.
\end{align*}  \end{theorem}

\begin{remark} The example using $G_6$ was studied in \cite{CynkHulek}, while all of the Hodge numbers using $G_6, H_6, L_6$ and $M_6$ can be found in \cite{FilippiniGarbagnati} as well as references therein. The pair $(h^{1,1},h^{2,1}) = (73,1)$ can be found in \cite{Garbagnati} and \cite{Kreuzer}, and a large set of pairs from a toric construction including $(35,11),(29,5)$ can be found in \cite{Kreuzer}. Lastly, as mentioned above, the pair (51,3) is the original Borcea construction \cite{Borcea}. These exhaust all the Calabi-Yau threefolds one can obtain from this construction, up to isomorphism, noting that any subgroup of $G_6$ that does not act trivially on one coordinate is isomorphic to one of the above subgroups.\end{remark}

\begin{remark} While the rigid examples cannot have Calabi-Yau mirror partners, all of the non-rigid examples have (topological) mirrors that have been constructed in the literature. All mirror pairs except $(1,73)$ can be found in the toric construction of \cite{Kreuzer}, while the last mirror can be found in \cite{BatyrevKreuzer} which constructs Calabi-Yau varieties and their mirrors via conifold transitions. It would be interesting to see if there is a relationship between the intermediate Jacobians in a mirror pair. \end{remark}


\begin{proof} As all the examples are similar, we only look at the cyclic example
\[ O_6 = \langle \iota_6^4\times\iota_6\times\iota_6 \rangle \]
which contains all the geometry necessary for the resolution of each example.

We must investigate the fixed points under the action of each element of this group, so we break things up into steps.

By continuously extending the appropriate automorphisms we have a birational diagram
\[ \xymatrix{ E_6^3 \ar[r] \ar[d] & E_6^3/O_6 & \ar[l] \ar[d] \widetilde{E_6^3/O_6} \\
              E_6^3/\langle \iota_6^2\times \iota_6^2\times \iota_6^2\rangle & \ar[l] \widetilde{E_6^3/\langle \iota_6^2\times \iota_6^2\times \iota_6^2\rangle} \ar[r] \ar[dr] & \left(\widetilde{E_6^3/\langle \iota_6^2\times \iota_6^2\times \iota_6^2\rangle}\right)/\langle \widetilde{\iota_6^4\times\iota_6\times\iota_6} \rangle \\
						& & \left(\widetilde{E_6^3/\langle \iota_6^2\times \iota_6^2\times \iota_6^2\rangle}\right)/\langle \widetilde{\id\times\iota_6^3\times\iota_6^3} \rangle \ar@{=}[u] } \]
with which we can resolve our threefold in straightforward manner. Note that at each step we will be blowing up fixed points or fixed curves, so the Hodge numbers $h^{i,0}$ do not change and our resolution will be crepant.

\subsubsection*{Step 1.} The K\"unneth formula gives
\[ h^{1,1}(E_6^3/\langle \iota_6^2\times \iota_6^2\times \iota_6^2\rangle) = 9, \quad\text{and}\quad h^{2,1}(E_6^3/\langle \iota_6^2\times \iota_6^2\times \iota_6^2\rangle) = 0, \]
and the resolution involves blowing up the 27 fixed points. Using \cite{SGA5} to compute the cohomology of a resolution, or (picking a model, and) looking at an affine patch explicitly and seeing the induced action on the blownup $\Pbb^2$ is trivial, we find
\[ h^{1,1}(\widetilde{E_6^3/\langle \iota_6^2\times \iota_6^2\times \iota_6^2\rangle}) = 36, \quad\text{and}\quad h^{2,1}(\widetilde{E_6^3/\langle \iota_6^2\times \iota_6^2\times \iota_6^2\rangle}) = 0. \]

\subsubsection*{Step 2.} We now quotient this resolution by (continuous extensions of) the remaining elements in $O_6$ to see what $(1,1)$-classes remain. Note that only 5 classes from the K\"unneth formula are preserved. Furthermore, the action of $\id\times\iota_6^3\times\iota_6^3$ identifies many of the 27 exceptional divisors from the previous blowup, so that
\[ h^{1,1}((\widetilde{E_6^3/\langle \iota_6^2\times \iota_6^2\times \iota_6^2\rangle})/O_6) = 20, \quad\text{and}\quad h^{2,1}((\widetilde{E_6^3/\langle \iota_6^2\times \iota_6^2\times \iota_6^2\rangle})/O_6) = 0. \]

\subsubsection*{Step 3.} The final resolution is now two separate blowups. We start with the three codimension 3 subvarieties under $\iota_6^4\times\iota_6\times\iota_6$, and then the six codimension 2 subvarieties from $\id\times\iota_6^3\times\iota_6^3$. If we denote by $O$ the identity of $E_6$ as a group, the point $(O,O,O)$ is fixed by each automorphism in $O_6$. After the blowup by $\iota_6^2\times\iota_6^2\times\iota_6^2$ we have a $\Pbb^2$ lying over this point, but this is not fixed by $\iota_6^4\times\iota_6\times\iota_6$ or $\id\times\iota_6^3\times\iota_6^3$. Instead, only the $\Pbb^1$ from the latter two coordinates is fixed, so the fixed locus includes a $\Pbb^1$ as well as the remaining fixed points. There are four fixed points $O,P_1,P_2,P_3$ on $E_6$ under the involution $\iota_6^3$, but $\iota_6^2(P_1) = \iota_6^2(P_2) = \iota_6^2(P_3)$. Hence, the fixed locus of (the continuous extension of) $\id\times\iota_6^3\times\iota_6^3$ on
\[ \widetilde{(E_6^3/\langle \iota_6^2\times\iota_6^2\times\iota_6^2\rangle)}/O_6 \]
is the $\Pbb^1$ above, as well as the (genus 1) fixed curves
\[ E_6\times O\times P_1, \]
\[ E_6\times P_1\times O, \]
\[ E_6\times P_1\times P_1, \]
\[ E_6\times P_1\times P_2, \]
\[ E_6\times P_1\times P_3. \]
Each of these resolved after a single blowup, and so a crepant resolution of $E_6^3/O_6$ has Hodge numbers
\[ h^{1,1}(\widetilde{E_6^3/O_6}) = 29, \qquad \text{and} \qquad h^{2,1}(\widetilde{E_6^3/O_6}) = 5, \]
as desired. \end{proof}

Note that for each subgroup $H$ of $G_3$, the variety $E_3^3/H$ is isomorphic to $E_6^3/J$ for some subgroup $J$ of $G_6$, so the above covers all the Calabi-Yau threefolds that arise using $E_3$ and $G_3$ as well.

All of these Hodge pairs, except for $(36,0)$, can be found using a generalization of a construction studied by Borcea \cite{Borcea} and Voisin \cite{Voisin}. For example, with the pair $(h^{1,1},h^{2,1}) = (29,5)$ we can look at the the action of
\[ \langle \iota_6^4\times\iota_6\times\iota_6 \rangle = \langle \iota_6^4\times\iota_6\times\iota_6, \id\times\iota_6^3\times\iota_6^3 \rangle \]
first as a ``birational Kummer construction'' taking the quotient of $E_6^3$ by $\id\times\iota_6^3\times\iota_6^3$, and then noting the induced action of $\iota_6^4\times\iota_6\times\iota_6$ is simply $\iota_3\times\iota_3\times\iota_3$ which is a generalized Borcea-Voisin threefold. See \cite{Garbagnati}.

Similarly, one can find the Calabi-Yau threefolds using this construction with $E_4^3$.

\begin{theorem} \label{aut4} Consider the groups of automorphisms
\begin{align*}
 G_4 &= \langle \iota_4\times \iota_4^3\times \id, \iota_4\times\id\times\iota_4^3 \rangle, && H_4 = \langle \iota_4^2\times\iota_4^2\times\id, \iota_4^2\times\id\times\iota_4^2\rangle, \\
 I_4 &= \langle \iota_4\times\iota_4\times\iota_4^2, \iota_4\times\iota_4^3\times\id \rangle, && J_4 = \langle \iota_4\times\iota_4\times\iota_4^2 \rangle,
\end{align*}
acting on the abelian threefold $E_4^3$. Crepant resolutions of the respective quotients are Calabi-Yau threefolds with Hodge numbers
\begin{align*}
 h^{1,1}(\widetilde{E_4^3/G_4}) &= 90, & h^{2,1}(\widetilde{E_4^3/G_4}) &= 0, \\
 h^{1,1}(\widetilde{E_4^3/H_4}) &= 51, & h^{2,1}(\widetilde{E_4^3/H_4}) &= 3, \\
 h^{1,1}(\widetilde{E_4^3/I_4}) &= 61, & h^{2,1}(\widetilde{E_4^3/I_4}) &= 1, \\
 h^{1,1}(\widetilde{E_4^3/J_4}) &= 31, & h^{2,1}(\widetilde{E_4^3/J_4}) &= 7.
\end{align*} \end{theorem}

\begin{remark} Again, this exhausts all the possible Calabi-Yau threefolds arising from this generalized Borcea construction with $E_4$, up to isomorphism, as any subgroup of $G_4$ for which a crepant resolution of the quotient is Calabi-Yau is isomorphic to one of the above groups. \end{remark}

The example with $G_4$ is seen in \cite{CynkHulek}, while $I_4$ is studied in \cite{Garbagnati} as a generalized Borcea-Voisin construction, and $H_4$ is the Borcea construction again, so we will only prove the result for the cyclic example with $J_4$.

\begin{proof} The main difference between this and the previous construction with $E_6^3$ is, instead of combining the automorphisms of order 2 and 3, we now have the fixed points satisfying
\begin{align}
 E_4^{\iota_4} \subset E_4^{\iota_4^2}. \label{fixed}
\end{align}
In particular, we may denote the fixed points of the involution $\iota_4^2$ by $P_0,P_1,P_2$ and $P_3$, where $\iota_4(P_2) = P_3$ and $\iota_4(P_3) = P_2$, while $\iota_4$ fixes both $P_0$ and $P_1$.


To compute the Hodge numbers, we start by noting the K\"unneth formula gives
\[ h^{1,1}(E_4^3/\langle \iota_4\times\iota_4\times\iota_4^2 \rangle) = 5 \qquad h^{2,1}(E_4^3/\langle \iota_4\times\iota_4\times\iota_4^2 \rangle) = 1. \]

Now, when blowing up each fixed point of $\iota_4\times\iota_4\times\iota_4^2$ we get an exceptional $\Pbb^2$, and $\iota_4^2\times\iota_4^2\times\id$ fixes a $\Pbb^1$ on this exceptional divisor, so using \eqref{fixed} we get a fairly simple crepant resolution by first blowing up the 16 fixed points to get a threefold $Y$, and then the remaining fixed curves.
\[ \xymatrix{ & \widetilde{E_4^3/J_4} \ar[d] \\
              & Y \ar[d] \\
 E_4^3 \ar[r] & E_4^3/\langle \iota_4\times \iota_4\times\iota_4^2 \rangle } \]
From here, each fixed $P_j\times P_k \times E_4$ under $\iota_4^2\times \iota_4^2\times\id$, where $0\leq j,k \leq 1$, corresponds to a rational curve in the quotient $Y$, while the other twelve $P_j\times P_k\times E_4$ are identified in pairs under the action of the $\iota_4$ on the first two coordinates, and are not rational curves in $Y$. Blowing up these 10 fixed curves, we find
\[ h^{1,1}(\widetilde{E_4/J_4}) = 5 + 16 + 10 = 31 \qquad \text{and} \qquad h^{2,1}(\widetilde{E_4/J_4}) = 1 + 6 = 7, \]
the desired Hodge numbers. \end{proof}

\begin{remark} The Hodge pairs $(90,0)$ and $(61,1)$ can again be found in \cite{Garbagnati} while $(61,1)$ also arises from a toric construction in \cite{Kreuzer}, as well as $(31,7)$, and these both have topological mirrors in \cite{BatyrevKreuzer}. The Borcea example with $(51,3)$ completes the list once more. \end{remark}

\section{Modularity}

We say a variety $X/\Qbb$ of dimension $d$ is \textit{modular} if its $L$-series, the $L$-series of the $d$-th $\ell$-adic cohomology of $X$, is the $L$-series of a newform $f = \sum a_nq^n$ of weight $d+1$ on $\Gamma_0(N)$ for some $N$. Here, a newform is taken to mean a normalized ($a_1 = 1$) eigenform on $\Gamma_0(N)$ that is not induced by a cusp form on $\Gamma_0(N')$ for any smaller $N'\mid N$. We say the newform $f$ has CM by a quadratic number field $K$ if $a_n = \varphi_K(n)a_n$ for almost all $n$, where $\varphi_K$ is the non-trivial Dirichlet character associated to $K$.

All rigid Calabi-Yau threefolds defined over $\Qbb$ are modular by \cite{Dieulefait} or \cite{GouveaYui}, in particular, all of the rigid Calabi-Yau threefolds in our construction. One can use Serre's conjecture to determine the associated newforms (hence $L$-series) as the (now proven) conjecture gives a bound on the level of the newform. It is a small product of powers of the primes of bad reduction of the corresponding threefold. However, this method is not practical when there are large primes of bad reduction, as one must construct spaces of cusp forms of high level and search for the correct newform. We will be able to avoid this by computing the $L$-series without the a priori knowledge of modularity, via a method shown to us by Hector Pasten \cite{Pasten}, taking advantage of the elliptic curves/abelian threefold structure of our varieties.

We will need one lemma about Hecke characters of imaginary quadratic number fields before approaching this computation, and for the convenience of the reader, we recall some basic facts about these characters. For any quadratic imaginary number field $K$ with ideal $\mfrak$ we say a Hecke character $\chi$ of modulus $\mfrak$ and infinity type $c$ is a homomorphism
\[ \chi : I_\mfrak \to \Cbb^\times \]
on fractional ideals of $K$, relatively prime to $\mfrak$, given by setting
\[ \chi(a\Ocal_K) = a^c \]
for all $a\in K^\times$ with $a\equiv 1\pmod \mfrak$. We may then extend it by setting it equal to 0 for any fractional ideal not coprime to $\mfrak$.

The $L$-series of $\chi$ is given by the product over all prime ideals $\pfrak$ of $K$
\[ L(s,\chi) = \prod_\pfrak \left(1 - \chi(\pfrak)N(\pfrak)^{-s}\right)^{-1} \]
where $N(\pfrak)$ is the norm of $\pfrak$. Hecke's insight was to associate a newform to $\chi$: the inverse Mellin transform
\[ f_\chi := \sum_\afrak \chi(\afrak)q^{N(\afrak)} \]
of the $L$-series is an eigenform of weight $c+1$, level $\Delta_K N(\mfrak)$ where $\Delta_K$ is the absolute value of the discriminant of $K$, and Nebentypus character
\[ \eta(n) := \frac{\chi(n\Ocal_K)}{n^c}. \]
Note that $\eta(n) = 0$ when $n$ is not coprime to the norm of the modulus, $N(\mfrak)$.

\begin{lemma} Let $\psi$ be a Hecke character of infinity type $c$ of an imaginary quadratic field $K$ and suppose its associated newform $f_\psi$ has trivial Nebentypus. Suppose that we have Fourier $q$-expansions
\[ f_\psi = \sum_{n=1}^\infty a_nq^n \qquad f_{\psi^3} = \sum_{n=1}^\infty b_nq^n. \]
Then
\[ b_p = a_p^3 - 3pa_p \]
for any prime $p$ not ramified in $K$. \label{tracelemma}\end{lemma}

\begin{proof} If $p = \pfrak\overline{\pfrak}$ splits in $K$, then $a_p = \psi(\pfrak) + \psi(\overline{\pfrak})$, so
\[ b_p = \psi(\pfrak)^3 + \psi(\overline{\pfrak})^3 = a_p^3 - 3pa_p. \]
If $p$ is inert, then $f_{\psi^3}$ has CM by $K$, so $a_p = b_p = 0$. \end{proof}

\begin{remark} The requirement that $f_\psi$ have trivial Nebentypus is subtle. Note that in, e.g., $\Qbb(\sqrt{-3})$, we can take a Hecke character $\psi$ such that $f_\psi$ has trivial Nebentypus, and a cubic twist $\psi\otimes \chi_3$ has associated newform $f_{\psi\otimes\chi_3}$ with non-trivial Nebentypus. If $f_\psi = \sum a_nq^n$ has trivial Nebentypus, then by the result we have $b_p = a_p^3 - 3pa_p$. On the other hand, the cubic twist has $(\psi\otimes\chi_3)^3 = \psi^3$, however a coefficient $a_p$ twisted by a third root of unity $\zeta_3$ satisfies
\[ (\zeta_3a_p)^3 - 3p(\zeta_3a_p) = a_p^3 - 3p\zeta_3a_p \neq b_p. \]
\end{remark}

In general one has a decomposition of the space of weight $k$ cusp forms on $\Gamma_1(N)$ over the Nebentypus characters $\epsilon$ modulo $N$,
\[ S_k(\Gamma_1(N)) = \bigoplus_\epsilon S_k(\Gamma_0(N),\epsilon). \]
In our case, all our newforms have real coefficients and we will only be interested in weight 2 or 4 newforms with CM by a quadratic imaginary number field $K$, so we recall
\begin{theorem}[Ribet \cite{Ribet} Prop 4.4, Thm 4.5] A newform has CM by a quadratic field $K$ if and only if it comes from a Hecke character of $K$. In particular, the field $K$ is imaginary and unique. \end{theorem}
Hence, as our newforms have CM and even weight $k$, any Nebentypus character $\epsilon$ satisfies $\epsilon(-1) = (-1)^k = 1$, so all of our newforms have trivial nebetypus, and the above lemma will apply.

We are now able to give a nice description of the modular forms associated to our rigid Calabi-Yau threefolds using the arithmetic of the underlying elliptic curves in the construction. For this we need to fix models of our elliptic curves over $\Qbb$, and their respective automorphisms. We fix the affine models
\[ E_3 : y^2 = x^3 - 1 \qquad \iota_3:(x,y)\mapsto(\zeta_3x,y), \]
\[ E_4 : y^2 = x^3 - x \qquad \iota_4:(x,y)\mapsto(-x,iy), \]
where $\zeta_3$ is a fixed primitive third root of unity.

Note that the orbits of each of the quotients in Theorems \ref{aut6} and \ref{aut4} are fixed under the action of $G_\Qbb := \Gal(\overline{\Qbb}/\Qbb)$, so each of the threefolds in this construction is defined over $\Qbb$. 

As our threefolds require the CM automorphisms that are not defined over every $\Qbb_\ell$, we will always work with a base extension to $\overline{\Qbb_\ell}$. We will write
\[ L(X,s) := L(H_\ell^3(\overline{X}),s) \]
where $H_\ell^k(\overline{X}) = \Het^k(X\otimes_{\Zbb_\ell}\overline{\Qbb_\ell},\overline{\Qbb_\ell})$.

\begin{theorem} Let $H$ be a subgroup of $G_3$ such that $X_3$, a crepant resolution of $E_3^3/H$, is a 
rigid Calabi-Yau threefold defined over $\Qbb$. We have
\[ L(X_3,s) = L(s,\chi_3^3) \]
where $\chi_3$ is the Hecke character of $E_3$, i.e., such that
\[ L(E_3,s) = L(s,\chi_3). \]
\end{theorem}

\begin{proof} With the rigid cases the resolution does not add any classes to the middle cohomology, and we have
\[ H_\ell^3(\overline{X_3}) \simeq H_\ell^3(\overline{E_3}^3/H) \simeq H_\ell^3(\overline{E_3}^3)^H. \]
The K\"unneth formula gives
\[ H_\ell^3(\overline{E_3}^3)^H = (H_\ell^1(\overline{E_3})\otimes H_\ell^1(\overline{E_3})\otimes H_\ell^1(\overline{E_3}))^H, \]
which is 2-dimensional. To be explicit, as
\[ H_\ell^1(\overline{E_3}) \simeq V_\ell(E_3) := T_\ell(E_3)^\vee \otimes \overline{\Qbb_\ell} \]
as Galois modules, where $T_\ell(E_3)$ is the Tate module of $E_3$, we will study the Galois representation on
\[ \Aut_{\overline{\Qbb_\ell}}(( V_\ell(E_3)\otimes V_\ell(E_3)\otimes V_\ell(E_3))^H ). \]

To understand the action of Frobenius under the Galois representation, we start by investigating the action on $E_3$. Denote the automorphism $\iota_3$ by $[\zeta_3]$ for notational convenience (the action of $\iota_3^2$ is then $[\zeta_3^2]$ and we may compute cleanly with the eigenvalues). This induces a non-trivial action $[\zeta_3]_*$ on the Tate module $V_\ell(E_3)$ with characteristic polynomial $T^2+T+1$. The eigenvalues of this action are thus the distinct primitive third roots of unity.

For any $\sigma\in G_\Qbb$, and any $(x,y)\in E_3(\overline{\Qbb})$, we have
\begin{align}
 \sigma([\zeta_3](x,y)) = \sigma((\zeta_3x,y)) = (\sigma(\zeta_3)\sigma(x),\sigma(y)) = [\sigma(\zeta_3)]\sigma((x,y)). \label{points3}
\end{align}
For any vector $v\in V_\ell(E_3)$ such that $\sigma_*(v)$ is an eigenvector of $[\zeta_3]_*$, we then have
\begin{align*}
 \zeta_3 \sigma_*(v) &= \sigma_*(\zeta_3v) \\
                     &= \sigma_*([\zeta_3]_*(v)) \\
										 &= (\sigma\circ[\zeta_3])_*(v) \\
										 &= ([\sigma(\zeta_3)]\circ\sigma)_*(v) \qquad \text{by}\ \eqref{points3} \\
										 &= [\sigma(\zeta_3)]_* \sigma_*(v).
\end{align*}
Hence, taking complex conjugates of both sides if necessary, $\sigma_*(v)$ is in the $\sigma(\zeta_3)$-eigenspace of $[\zeta_3]_*$. In particular, if $c\in G_\Qbb$ is complex conjugation, then $w := c_*(v)$ gives a $\zeta_3^2$-eigenvector for $[\zeta_3]_*$ under $\sigma_*$.

Let $\chi: G_\Qbb\to\overline{\Qbb_\ell}^\times$ be the non-trivial Dirichlet character on $\Qbb(\zeta_3)$. Fix a prime $p\neq 2,3$, so that $E_3$ has good reduction at $p$. If $\chi(\Frob_p) = 1$, then the above shows that $
(\Frob_p)_*(v)$ is a $\zeta_3$-eigenvector for $(\Frob_p)_*$ and $w$ gives another eigenvector so that the induced action of Frobenius on $V_\ell(E_3)$ with the basis $v,w$ is given by a matrix
\[ \begin{pmatrix} \alpha_p & 0 \\ 0 & \beta_p \end{pmatrix} \]
where $\alpha_p,\beta_p$ are the eigenvalues of $(\Frob_p)_*$.

On the other hand, if $\chi(\Frob_p) = -1$, then $(\Frob_p)_*(v)$ is a $\zeta_3^2$-eigenvector of $[\zeta_3]_*$, and $(\Frob_p)_*(w)$ is a $\zeta_3$-eigenvector of $[\zeta_3]_*$, so the action in the basis $v,w$ is given by
\[ \begin{pmatrix} 0 & h_p \\ k_p & 0 \end{pmatrix} \]
for some $h_p,k_p$ such that $h_pk_p = -p$.

On $(H_\ell^1(\overline{E_3})^{\otimes 3})^H$, we know the pure tensors
\[ v\otimes v\otimes v\quad \text{and}\quad w\otimes w\otimes w \]
are fixed by $H$, and span the space, hence are a basis. If we denote the Galois representation by
\[ \rho_3 : G_\Qbb\to \Aut_{\overline{\Qbb_\ell}}(( V_\ell(E)\otimes V_\ell(E)\otimes V_\ell(E))^H ), \]
then we have two possibilities for $\rho_3(\Frob_p)$. If $\chi(\Frob_p) = 1$, the action of $\rho(\Frob_p)$ is given by the matrix
\[ \begin{pmatrix} \alpha_p^3 & 0 \\ 0 & \beta_p^3 \end{pmatrix} \]
while if $\chi(\Frob_p) = -1$ the action is given by
\[ \begin{pmatrix} 0 & h_p^3 \\ k_p^3 & 0 \end{pmatrix}. \]
Now, as $\alpha_p,\beta_p = \pm i\sqrt{p}$ when $\chi(\Frob_p) = -1$, we simply have
\[ \tr(\rho(\Frob_p)) = \alpha_p^3 + \beta_p^3 = (\alpha_p+\beta_p)^3 - 3p(\alpha_p+\beta_p). \]
Lemma \ref{tracelemma} completes the proof. \end{proof}

The main interesting piece of arithmetic that comes from working over $\Qbb$ instead of $\Cbb$ is that our elliptic curves are no longer unique up to isomorphism, and we can investigate what occurs if we pick another model. Using twists of the elliptic curves and proceeding with the construction, we get an appropriate twist of the $L$-series of the threefold. In this sense, this defines twists of our threefolds, as in \cite{GouveaKimingYui}. Let $D$ be an integer, and denote by $E_3(D)$ the curve
\[ E_3(D) : y^2 = x^3 - D. \]
Then $E_3 = E_3(1)$, and if $D$ is square-free, cube-free or fourth-power-free, the curve $E_3(D)$ is a sextic, cubic or quadratic twist of $E_3$ respectively. The action of Frobenius on $E_3(D)$ is the action of the Frobenius on $E_3$ twisted by $\psi_D$, a non-trivial sexitc, cubic or quadratic Dirichlet character of $\Qbb(\sqrt{D})$ respectively, when $D\neq 1$.
\begin{remark} If $1 \neq D = (-1)^k p_1^{k_1}p_2^{k_2}\cdots p_n^{k_n}$ is a prime factorization of $D$, then
\[ \psi_D = \psi_{p_1}^{k_1}\psi_{p_2}^{k_2}\cdots \psi_{p_n}^{k_n}, \]
so if $D$ is not fourth-power-free there are fourth-power-free integers $k_j' \leq k_j$ for $1\leq j\leq n$ such that $D' = (-1)^k p_1^{k_1'}p_2^{k_2'}\cdots p_n^{k_n'}$ gives the same twist, i.e., $E_3(D') \simeq E_3(D)$ over $\Qbb$. As such, we do not concern ourselves with always ensuring our twists are fourth-power-free. \end{remark}

On a crepant resolution of
\[ (E_3(D)\times E_3(D)\times E_3(D))/H \]
we have
\[ \tr(\rho(\Frob_p)) = \psi_{D^3}^3(\Frob_p)(\alpha_p^3 + \beta_p^3). \]
This extends to the case where we do not twist all three curves by the same $D$. Suppose $D_1,D_2$ and $D_3$ are not necessarily equal. On a crepant resolution of
\[ E_3^3(D_1,D_2,D_3) := (E_3(D_1)\times E_3(D_2^)\times E_3(D_3^))/H \]
we have
\[ \tr(\rho(\Frob_p)) = \psi_{D_1D_2D_3}(\Frob_p)(\alpha_p^3 + \beta_p^3). \]

With this notation set up, we have
\begin{theorem} \label{mod3} Let $H$ be a subgroup of $G_3$ such that a crepant resolution of $E_3^3(D_1,D_2,D_3)/H$ is a rigid Calabi-Yau threefold defined over $\Qbb$. If we denote this resolution by $Y_3$, then
\[ L(Y_3,s) = L(s,\chi_3^3) \]
where $\chi_3$ is the Hecke character such that
\[ L(E_3(D_1D_2D_3),s) = L(s,\chi_3). \]
\end{theorem}

In a similar fashion, we define the twists
\[ E_4(D) : y^2 = x^3 - Dx. \]
If $D$ is square-free this is a biquadratic twist of $E_4$, and otherwise if $D$ if cube-free this is a quadratic twist. The twisted threefolds are
\[ E_4^3(D_1,D_2,D_3) := (E_4(D_1)\times E_4(D_2)\times E_4(D_3))/G_4 \]
and we again let $E_4(1)$ denote $E_4$ itself.
\begin{theorem} Let $D_1,D_2$ and $D_3$ be integers such that a crepant resolution $Y_4$ of $E_4^3(D_1,D_2,D_3)/G_4$ is a rigid Calabi-Yau threefold defined over $\Qbb$. We have
\[ L(Y_4,s) = L(s,\chi_4^3) \]
where $\chi_4$ is the Hecke character such that
\[ L(E_4(D_1D_2D_3),s) = L(s,\chi_4). \]
\end{theorem}

\begin{proof} As a (non-trivial) twist only multiplies the trace of the action of Frobenius by the respective quadratic or biquadratic character as above, we need only show the $L$-series are as described in the case without twisting the underlying elliptic curves.

The induced action on the (extended) Tate module $V_\ell(E_4)$ is the only difference here. If we denote the action on $E_4$ by $[i]$, we have $[i]^2 = [-1]$, so the eigenvalues of the action are $\pm i$. For any $\sigma\in G_\Qbb$ and $(x,y)\in E_4(\overline{\Qbb})$ we have
\[ \sigma([i](x,y)) = [\sigma(i)]\sigma((x,y)). \]
Hence, for an eigenvector $\sigma_*(v)$ we have
\begin{align*}
 i\sigma_*(v) &= \sigma_*([i]_*(v)) \\
              &= (\sigma\circ [i])_*(v) \\
							&= ([\sigma(i)]\circ \sigma)_*(v) \\
							&= [\sigma(i)]_*(\sigma_*(v)) \\
							&= \chi(\sigma)[i]_*(\sigma_*(v)) 
\end{align*}
where $\chi$ is the non-trivial Dirichlet character on $\Qbb(i)$. We again find that if $c$ denotes complex conjugation, then $w = c_*(v)$ gives a $(-i)$-eigenvector of $[i]_*$ under $\sigma_*$ so that $v,w$ gives a basis for $V_\ell(E_4)$. The computation of the action of Frobenius again divides into the two cases where $\chi(\Frob_p) = \pm 1$, and is otherwise as above with $E_3$. \end{proof}

With a little adjustment to the action on the Tate module of $E_3$ we can find the $L$-series associated to the rest of the constructions of rigid Calabi-Yau threefolds using the automorphism of order 6 on $E_3$. (Which reproves Theorem \ref{mod3}.)

\begin{theorem} Let $J$ be a subgroup of $G_6$ such that a crepant resolution of $E_6^3(D_1,D_2,D_3)/J$ is a rigid Calabi-Yau threefold defined over $\Qbb$. If we denote this resolution by $Y_6$, then
\[ L(Y_6,s) = L(s,\chi_6^3) \]
where $\chi_6$ is the Hecke character such that
\[ L(E_6(D_1D_2D_3),s) = L(s,\chi_6). \]
\end{theorem}

\section{Intermediate Jacobians}

In this section we compute the intermediate Jacobians of the rigid Calabi-Yau threefolds described above. For each rigid Calabi-Yau threefold $X$, its intermediate Jacobian simply
\[ J(X) = H^{0,3}(X)/H^3(X,\Zbb), \]
a 1-torus. The method we use to compute these tori is to recognize how to reconstruct the torus structure of an elliptic curve from a analogous quotient of its cohomology groups, extending Roan's work on Kummer threefolds in \cite{Roan}. We then analyze the construction of the intermediate Jacobians as complex varieties more closely to determine a $\Qbb$-model in each case.

\subsection{Complex torus structure}

For any $\tau = \alpha + \beta i$ in the upper half plane we have the elliptic curve $E_\tau = \Cbb/\langle 1,\tau\rangle$, with uniformizing parameter $z = x+iy$. Thus, as complex tori, $E_i = E_4$ and $E_\zeta = E_3$ where $\zeta = e^{2\pi i/3}$.

Translations by $1,\tau$ in $\Cbb$ give rise to a basis $e,f\in H_1(E_\tau,\Zbb)$ so that
\[ \begin{pmatrix} e \\ f \end{pmatrix} = \begin{pmatrix} 1&0 \\ \alpha&\beta \end{pmatrix}\begin{pmatrix} \partial_x \\ \partial_y \end{pmatrix}, \]
where $\partial_x$ and $\partial_y$ are a basis for $H_1(E_\tau,\Cbb)$ corresponding to the uniformizing parameter. Taking duals in cohomology then gives a basis $e^*,f^*\in H^1(E_\tau,\Zbb)$ such that
\[ \begin{pmatrix} e^* & f^* \end{pmatrix}\begin{pmatrix} 1&0 \\ \alpha&\beta \end{pmatrix} = \begin{pmatrix} dx & dy \end{pmatrix}, \]
where $dz = dx + idy$ is the holomorphic $1$-form on $E$ corresponding to the uniformizing parameter. Thus, we have
\[ 2e^* = \left(1 + \frac{\alpha}{\beta}i\right)dz + \left(1 - \frac{\alpha}{\beta}\right)d\overline{z}, \]
\[ 2f^* = \frac{i}{\beta}(d\overline{z} - dz), \]
so using $d\overline{z}/2$ as a generator for $H^{0,1}(E_\tau)$ we have
\[ H^{0,1}(E_\tau)/H^1(E_\tau,\Zbb) = \Cbb/\left(\left(1 - \frac{\alpha}{\beta}i\right)\Zbb \oplus \frac{i}{\beta}\Zbb\right). \]
Applying the homothety given by multiplication by $\beta/i$ we have
\[ H^{0,1}(E_\tau)/H^1(E_\tau,\Zbb) \simeq \Cbb/(\Zbb\oplus (\alpha + i\beta)\Zbb) = E_\tau. \]

To mimic this procedure for the intermediate Jacobians of our rigid Calabi-Yau threefolds we need to find a basis for the integral cohomology and find a period relation with a basis for the complex cohomology.

As $H^3(E_\tau^3/G_\tau,\Zbb) \simeq H^3(E_\tau^3,\Zbb)^{G_\tau}$, we may start with the simpler $H^3(E_\tau^3,\Zbb)$. As above, let $(z_1,z_2,z_3)\in \Cbb^3$ be uniformizing coordinates for $E_\tau^3$, with $z_k = x_k + iy_k$ corresponding to the $k$-th coordinate. Let $\{e_k,f_k\}$ be a basis for $H_1(E_\tau,\Zbb)$, for $k=1,2,3$, with dual bases $e_k^*,f_k^*$. This naturally gives bases for both the homology $H_3(E_\tau^3,\Zbb)$ and the cohomology $H^3(E_\tau^3,\Zbb)$, with the relations
\[ \begin{pmatrix} e_k \\ f_k \end{pmatrix} = \begin{pmatrix} 1 & 0 \\ \alpha & \beta \end{pmatrix}\begin{pmatrix} \partial_{x_k} \\ \partial_{y_k} \end{pmatrix} \qquad\qquad \begin{pmatrix} e_k^* & f_k^* \end{pmatrix}\begin{pmatrix} 1 & 0 \\ \alpha & \beta \end{pmatrix} = \begin{pmatrix} dx_k & dy_k \end{pmatrix} \]
for each $k$. The holomorphic threeform $\Omega = dz_1\wedge dz_2\wedge dz_3$ can be written using the dual basis to give the desired period relation for the intermediate Jacobian. Indeed, as
\begin{align*}
 dz_1\wedge dz_2\wedge dz_3 &= dx_1\wedge dx_2\wedge dx_3 + i(dx_1\wedge dx_2\wedge dy_3 + dx_1\wedge dy_2\wedge dx_3 + dy_1\wedge dx_2\wedge dx_3) \\
 &\ - i(dx_1\wedge dy_2\wedge dy_3 + dy_1\wedge dx_2\wedge dy_3 + dy_1\wedge dy_2\wedge dx_3) - i(dy_1\wedge dy_2\wedge dy_3),
\end{align*}
we may write $\Omega = \Real(\Omega) + i\Img(\Omega)$, where
\begin{align*}
 \Real(\Omega) &= dx_1 \wedge dx_2 \wedge dx_3 - dx_1 \wedge dy_2 \wedge dy_3 - dy_1 \wedge dx_2 \wedge dy_3 - dy_1 \wedge dy_2 \wedge dx_3 \\
 &= e_1^* \wedge e_2^* \wedge e_3^* + \alpha(e_1^*\wedge e_2^*\wedge f_3^* + e_1^*\wedge f_2^*\wedge e_3^* + f_1^*\wedge e_2^*\wedge e_3^*) \\
 &\ + (\alpha^2 - \beta^2)(e_1^*\wedge f_2^*\wedge f_3^* + f_1^*\wedge e_2^*\wedge f_3^* + f_1^*\wedge f_2^* \wedge e_3^*) + (\alpha^3 - 3\alpha\beta^2)f_1^*\wedge f_2^*\wedge f_3^*, 
\end{align*}
and
\begin{align*}
 \Img(\Omega) &= dx_1 \wedge dx_2 \wedge dy_3 + dx_1 \wedge dy_2 \wedge dx_3 + dy_1 \wedge dx_2 \wedge dx_3 - dy_1 \wedge dy_2 \wedge dy_3 \\
 &= \beta(e_1^*\wedge e_2^*\wedge f_3^* + e_1^*\wedge f_2^*\wedge e_3^* + f_1^*\wedge e_2^*\wedge e_3^*) + 2\alpha\beta(e_1^*\wedge f_2^*\wedge f_3^* + f_1^*\wedge e_2^*\wedge f_3^* + f_1^*\wedge f_2^*\wedge e_3^*) \\
 &\ +(3\alpha^2\beta - \beta^3)f_1^*\wedge f_2^*\wedge f_3^*.
\end{align*}
With this we can now compute the intermediate Jacobians of our rigid Calabi-Yau threefolds. We start with the simpler case of $X_4$, a crepant resolution of $E_4^3/G_4$.

To compute the intermediate Jacobian of $X_4$, we have each underlying elliptic curve in the product having complex period $\tau = i$, so
\[ \Real(\Omega_4) = e_1^*\wedge e_2^*\wedge e_3^* - e_1^*\wedge f_2^*\wedge f_3^* - f_1^*\wedge e_2^*\wedge f_3^* - f_1^*\wedge f_2^*\wedge e_3^*, \]
\[ \Img(\Omega_4) = e_1^*\wedge e_2^*\wedge f_3^* + e_1^*\wedge f_2^*\wedge e_3^* + f_1^*\wedge e_2^*\wedge e_3^* - f_1^*\wedge f_2^*\wedge f_3^* \]
where $\Omega_4 \in H^{3,0}(X_4)$. These give a basis for $H^3(X_4,\Cbb)$. For the left hand side of the period relation, we choose classes $A_4 = \Real(\Omega_4)$ and $B_4 = \Img(\Omega_4)$ and note that $A_4,B_4\in H^3(X,\Zbb)$. We claim this is a basis for $H^3(X_4,\Zbb)$.

Indeed, suppose we have a basis $C,D$ of $H^3(X_4,\Zbb)$. As $A_4,B_4\in H^3(X_4,\Zbb)$, we know there is some matrix $M$ with \textit{integral} entries so that
\begin{align}
 \begin{pmatrix} A_4 \\ B_4 \end{pmatrix} = M\begin{pmatrix} C \\ D\end{pmatrix}. \label{periods}
\end{align}
Write
\[ M = \begin{pmatrix} a&b \\ c&d \end{pmatrix} \]
If $ad = bc$ then $A_4$ must be a multiple of $B_4$, which we know is not true from our explicit expressions above, so $a,b,c,d\in \Zbb$ and $ad\neq bc$, i.e., $M\in \GL_2(\Qbb)$. Moreover, as $C,D\in H^3(X_4,\Zbb)$ we may write
\[ C = qe_1^*\wedge e_2^*\wedge e_3^* + re_1^*\wedge e_2^*\wedge f_3^* + \ldots, \]
\[ D = se_1^*\wedge e_2^*\wedge e_3^* + te_1^*\wedge e_2^*\wedge f_3^* + \ldots, \]
for some $q,r,s,t\in \Zbb$. Multiplying all this out in \eqref{periods} then gives
\[ A_4 = (aq + bs)e_1^*\wedge e_2^*\wedge e_3^* + (ar + bt)e_1^*\wedge e_2^*\wedge f_3^* + \ldots, \]
\[ B_4 = (cq + ds)e_1^*\wedge e_2^*\wedge e_3^* + (cr + dt)e_1^*\wedge e_2^*\wedge f_3^* + \ldots, \]
and so, comparing with our expressions for $A_4$ and $B_4$ above, we have integral matrices such that
\[ \begin{pmatrix} a & b \\ c & d \end{pmatrix}\begin{pmatrix} q & r \\ s & t \end{pmatrix} = \begin{pmatrix} 1 & 0 \\ 0 & 1 \end{pmatrix}, \]
i.e., $M\in \GL_2(\Zbb)$. Hence, $A_4,B_4$ are a basis for $H^3(X_4,\Zbb)$.

Thus, the period relation for the intermediate Jacobian of $X_4$ is
\[ \begin{pmatrix} A_4 & B_4 \end{pmatrix} \begin{pmatrix} 1 & 0 \\ 0 & 1\end{pmatrix} = \begin{pmatrix} \Real(\Omega_4) & \Img(\Omega_4) \end{pmatrix} \]
so that
\[ 2A_4 = \Omega_4 + \overline{\Omega_4} \]
\[ 2B_4 = i(\overline{\Omega_4} - \Omega_4), \]
and using $\overline{\Omega_4}/2$ as a basis for $H^{0,3}(X_4)$, we have
\[ J(X_4) \simeq \Cbb/(\Zbb\oplus i\Zbb) = E_4. \]

\subsection{Model for the intermediate Jacobian $J(X_4)$ over $\Qbb$}

As we are not just interested in the complex torus structure of the intermediate Jacobian, we may not simply work up to homothety, and we must be more careful in how to recover not just the torus structure from the period relation, but the exact model defined over $\Qbb$. By the Uniformization Theorem, we know for any elliptic curve $E=\Cbb/\Lambda$, there is a $\lambda\in \Cbb^\times$ such that any particular model of $E$ corresponds uniquely to the torus $\Cbb/\lambda\Lambda$, the correspondence being
\[ E : y^2 = 4x^3 - \lambda^{-4}g_2(\Lambda)x - \lambda^{-6}g_3(\Lambda) \longleftrightarrow \Cbb/\lambda\Lambda, \]
where $g_2 = 60G_4$ and $g_3 = 140G_6$ with $G_{2k}$ the Eisenstein series of weight $2k$.

We are interested in using the computation above to recover a particular model, thus suppose we have some $E = \Cbb/\lambda\langle 1,\tau\rangle$. Translation by $1$ and $\tau = \alpha + i\beta$ no longer gives a basis for $H_1(E,\Zbb)$. We now get a basis using translation by $\lambda$ and $\lambda \tau$. Similarly, our integral classes are no longer $e^*$ and $f^*$, and our period relation is
\[ \begin{pmatrix} \frac{e^*}{\lambda} & \frac{f^*}{\lambda} \end{pmatrix}\begin{pmatrix} \lambda&0 \\ \lambda\alpha&\lambda\beta \end{pmatrix} = \begin{pmatrix} dx & dy \end{pmatrix} . \]
Hence, we recover the same relations as before
\[ 2 e^* = \left(1 + \frac{\alpha}{\beta}i\right)dz + \left(1 - \frac{\alpha}{\beta}i\right)d\overline{z}, \]
\[ 2 f^* = \frac{i}{\beta}(d\overline{z} - dz). \]
Using $d\overline{z}/2$ as a basis for $H^{0,1}(E)$, like above, gives
\[ H^{0,1}(E)/H^1(E,\Zbb) = \Cbb/\left(\left(1 - \frac{\alpha}{\beta}i\right)\Zbb \oplus \frac{i}{\beta}\Zbb\right) \]
which is \textit{not} $\Cbb/\lambda\langle 1,\tau\rangle$, the torus we started with. Instead, we must use the basis $d\overline{z}/(2\beta\lambda i)$ so that
\[ e^* = \lambda(\alpha - \beta i)\frac{dz}{2\beta \lambda i} - \lambda(\alpha + \beta i)\frac{d\overline{z}}{2\beta\lambda i}, \]
\[ f^* = \lambda\frac{dz}{2\beta\lambda i} - \lambda\frac{d\overline{z}}{2\beta\lambda i}. \]
Then we have
\[ H^{0,1}(E)/H^1(E,\Zbb) = \Cbb/\lambda(\Zbb\oplus (\alpha+i\beta)\Zbb) = E. \]
Thus for any $E = \Cbb/\lambda(\Zbb\oplus \tau\Zbb)$ we can recover the model using the period relation and basis $d\overline{z}/(2\beta\lambda^2i)$ of $H^{0,1}(E)$.

In our case of interest with $E_4:y^2 = x^3 - x$ we have (see \cite{Waldschmidt}) $E_4 = \Cbb/\lambda\langle 1,i\rangle$ with
\[ \lambda = \frac{\Gamma(\frac{1}{4})^2}{2\sqrt{2\pi}}. \]

\begin{remark} We note that $\lambda$ is transcendental, as $\Gamma(1/4)$ and $\sqrt{\pi} = \Gamma(1/2)$ are algebraically independent, \cite{Waldschmidt}. \end{remark}

On the threefold we have
\[ \begin{pmatrix} \frac{e_k^*}{\lambda} & \frac{f_k^*}{\lambda} \end{pmatrix}\begin{pmatrix} \lambda&0 \\ 0&\lambda \end{pmatrix} = \begin{pmatrix} dx_k & dy_k \end{pmatrix} \]
for the $k$-th component of $E_4^3$. Writing $\Omega_4 = \Real(\Omega_4) + i\Img(\Omega_4)$ and using the period relations for each of the underlying $E_4$ to write these in terms of the $e_k^*$ and $f_k^*$ we again find
\begin{align*}
 \Real(\Omega_4) &= A_4, \\
 \Img(\Omega_4) &= B_4.
\end{align*}
These classes are no longer integral, but as above our integral classes are $A_4/\lambda^3$ and $B_4/\lambda^3$ so that the period relation is
\[ \begin{pmatrix} \frac{A_4}{\lambda^3} & \frac{B_4}{\lambda^3} \end{pmatrix}\begin{pmatrix} \lambda^3 & 0 \\ 0 & \lambda^3 \end{pmatrix} = \begin{pmatrix} \Real(\Omega_4) & \Img(\Omega_4) \end{pmatrix}. \]

To get the correct model of the elliptic curve described by this relation, we must use the basis
\[ \frac{\overline{\Omega_4}}{2\beta\lambda^3 i} = \frac{\overline{\Omega_4}}{2\lambda^3 i} \]
for $H^{0,3}(X_4)$. Thus, writing
\[ A_4 = i\lambda^3 \left(\frac{\Omega_4}{2\lambda^3 i}\right) + i\lambda^3 \left(\frac{\overline{\Omega_4}}{2\lambda^3 i}\right) \]
\[ B_4 = \lambda^3\left(\frac{\Omega_4}{2\lambda^3 i}\right) - \lambda^3\left(\frac{\overline{\Omega_4}}{2\lambda^3 i}\right), \]
we find
\[ J(X_4) = \Cbb/\lambda^3(\Zbb\oplus i\Zbb) \neq E_4. \]
Even worse, the intermediate Jacobian is not even defined over a number field. However, using the basis $\overline{\Omega_4}/(2\lambda i)$ we recover the model of $E_4$ for the intermediate Jacobian, which is still quite natural. Thus, we have the following.

\begin{theorem} Let $D$ be a square-free integer, and $E_4(D^2)$ the elliptic curve with affine equation $y^2 = x^3 - D^2x$, a quadratic twist of $E_4$. Let $Y_4(D^2)$ be a crepant resolution of
\[ (E_4(D^2)\times E_4(D^2)\times E_4(D^2))/ G_4. \]
Then
\[ J(Y_4(D^2)) = E_4(D^2), \]
and hence there is a model for the intermediate Jacobian $J(Y_4(D^2))$ such that
\[ L(J(Y_4(D^2)),s) = L(s,\chi) \]
where $\chi$ is the Hecke character of $E_4(D^2)$, and
\[ L(Y_4(D^2),s) = L(s,\chi^3). \]
In particular, Conjecture \ref{Yui} is true for $Y_4(D^2)$.
\end{theorem}

\begin{remark} As our $\Qbb$-model for $J(Y_4(D^2))$ requires choosing a scaled basis for the complex cohomology, one could argue the result above is ad hoc and simply designed to get the result we desire. However, note that having fixed the underlying elliptic curves on the threefold, we have fixed the integral cohomology classes, and so the period relation cannot be scaled by a non-integer. The uniformization theorem tells us integral twists require non-integral homotheties, and so our period relation cannot be scaled in some way to give another model for $J(Y_4(D^2))$, defined over $\Qbb$. Choosing a different uniformizing parameter to get a different model defined over $\Qbb$ is then what appears ad hoc. Hence, while the particular model for $J(Y_4(D^2))$ depends on the choice of basis, the natural $\Qbb$-model does not. Now, as any CM elliptic curve has a model defined over a number field, our $\Qbb$-model seems the best fit. \end{remark}

Given this, one can ask if a period relation
\[ \begin{pmatrix} \frac{A_4}{\mu} & \frac{B_4}{\mu} \end{pmatrix}\begin{pmatrix} \mu & 0 \\ 0 & \mu \end{pmatrix} = \begin{pmatrix} \Real(\Omega_4) & \Img(\Omega_4) \end{pmatrix} \]
for the intermediate Jacobian of a rigid Calabi-Yau threefold always gives the $\Qbb$-model if we use $\overline{\Omega_4}/(2\sqrt[3]{\mu}i)$ as a basis, which is still quite natural. This, however, does not work in, e.g., the case where each underlying elliptic curve has a distinct twist. Indeed, consider real $\kappa_1,\kappa_2$ and $\kappa_3$ such that $\Cbb/\kappa_i\lambda\langle 1,\tau\rangle$ are distinct quadratic twists $E_4(D_1^2), E_4(D_2^2)$ and $E_4(D_3^2)$. Let $X$ be a crepant resolution of
\[ (E_1(D_1^2)\times E_2(D_2^2)\times E_3(D_3^2))/G_4. \]
The period relation for $J(X)$ is then
\[ \begin{pmatrix} \frac{A_4}{\kappa_1\kappa_2\kappa_3\lambda^3} & \frac{B_4}{\kappa_1\kappa_2\kappa_3\lambda^3} \end{pmatrix}\begin{pmatrix} \kappa_1\kappa_2\kappa_3\lambda^3 & 0 \\ 0 & \kappa_1\kappa_2\kappa_3\lambda^3 \end{pmatrix} = \begin{pmatrix} \Real(\Omega_4) & \Img(\Omega_4) \end{pmatrix} \]
and using $\overline{\Omega_4}/(2\sqrt[3]{\kappa_1\kappa_2\kappa_3\lambda^3} i)$ as a basis for $H^{0,3}(X)/H^3(X,\Zbb)$ then gives a model not defined over $\Qbb$, while the model
\[ J(X) : y^2 = x^3 - (D_1D_2D_3)^2x \]
comes from using $\overline{\Omega_4}/(2\kappa_1\kappa_2\kappa_3\lambda i)$ as a basis instead. To reconcile these differences, we move away from working entirely with the underlying elliptic curves, and instead move to focusing on the threefold. Note that
\[ (E_4(D_1^2)\times E_4(D_2^2)\times E_4(D_3^2))/G_4 \]
is birational to
\[ (E_4((D_1D_2D_3)^2)\times E_4\times E_4)/G_4, \]
and we may shift our computations using the underlying elliptic curves to exploit the quadratic twists of the threefold itself. In this setting, writing the quadratic twist of
\[ (E_4(D_1^2)\times E_4(D_2^2)\times E_4(D_3^2))/G_4 \]
as $D = D_1D_2D_3$, with $E_4(D) = \Cbb/\kappa\lambda\langle 1,i\rangle$, we have that $d\overline{\Omega}/(2\kappa\lambda i) \in H^{0,3}(X(D))$ used as a basis recovers
\[ E(D^2) : y^2 = x^3 - D^2x = x^3 - (D_1D_2D_3)^2x, \]
the appropriate model for the conjecture.

Unfortunately, when allowing biquadratic twists, the conjecture is no longer always true. If we have a biquadratic character $\psi$, then $\psi^3 = \overline{\psi}$, i.e., the character
\[ \overline{\psi}(z) = \overline{\psi(z)} \]
and so if we have a rigid Calabi-Yau threefold $Y$, with CM, satisfying the conjecture, and take a biquadratic twist by $\psi$, (e.g., twist only the first elliptic curve by $\psi$) then its intermediate Jacobian has a model with
\[ L(J(Y_\psi),s) = L(s,\chi\otimes\psi) \]
while
\[ L(Y_\psi,s) = L(s,\chi^3\otimes\psi) \neq L(s,\chi^3\otimes\overline{\psi}) = L(s,(\chi\otimes\psi)^3). \]
Twisting two of the elliptic curves by the biquadratic character will, however, give a quadratic twist of the rigid threefold, so this notion of twisting threefolds in \cite{GouveaKimingYui} is the natural setting to characterize when Yui's conjecture is true.

\begin{theorem} Let $D_1,D_2$ and $D_3$ be squarefree integers, and $E_4(D_j^{k_j})$ the elliptic curve with affine equation $y^2 = x^3 - D_j^{k_j}x$, where $k_j\in\{1,2\}$ for each $j$. Let $Y_4(D_1^{k_1},D_2^{k_2},D_3^{k_3})$ be a crepant resolution of
\[ (E_4(D_1^{k_1})\times E_4(D_2^{k_2})\times E_4(D_3^{k_3}))/ G_4. \]
Then
\[ J(Y_4(D_1^{k_1},D_2^{k_2},D_3^{k_3})) = E_4(D_1^{k_1}D_2^{k_2}D_3^{k_3}). \]
There is a model for the intermediate Jacobian such that
\[ L(J(Y_4(D_1^{k_1},D_2^{k_2},D_3^{k_3})),s) = L(s,\chi), \]
where $\chi$ is the Hecke character of $E_4(D_1^{k_1}D_2^{k_2}D_3^{k_3})$, and
\[ L(Y_4(D_1^{k_1},D_2^{k_2},D_3^{k_3}),s) = L(s,\chi^3). \]
if and only if $D_1^{k_2}D_2^{k_2}D_3^{k_2}$ is the square of an integer.
In particular, Conjecture \ref{Yui} is true only for quadratic twists of $\widetilde{E_4^3/G_4}$.
\end{theorem}

\subsection{Model for the intermediate Jacobian $J(X_6)$ over $\Qbb$}

Similarly for $E_6$, we have our exact model corresponding to the complex torus $\Cbb/\mu\Gamma$ where
\[ \mu = \frac{\Gamma(\frac{1}{3})^3}{2^{4/3}3^{1/2}\pi} \]
and $\Gamma = \Zbb \oplus \zeta_3\Zbb$ with $\zeta_3 = e^{2\pi i/3}$.

\begin{remark} Again, we have that $\mu$ is transcendental by \cite{Waldschmidt}. \end{remark}

Abusing notation, let $z_k = x_k + iy_k$ be the uniformizing parameter of the $k$-th elliptic curve on $E_6^3$. The period relation for the (complex) intermediate Jacobian of $X_6$ is
\[ \begin{pmatrix} A_6 & B_6 \end{pmatrix} \begin{pmatrix} 1 & 0 \\ -\frac{1}{2} & \frac{\sqrt{3}}{2} \end{pmatrix} = \begin{pmatrix} \Real(\Omega_6) & \Img(\Omega_6) \end{pmatrix} \]
where $\Omega_6 = dz_1\wedge dz_2\wedge dz_3$. This time we find
\[ \Real(\Omega_6) = A_6 -  \frac{1}{2}B_6, \]
\[ \Img(\Omega_6) = \frac{\sqrt{3}}{2}B_6, \]
with integral classes
\begin{align*}
 A_6 &= e_1^*\wedge e_2^*\wedge e_3^* - e_1^*\wedge f_2^*\wedge f_3^* - f_1^*\wedge e_2^*\wedge f_3^* - f_1^*\wedge f_2^*\wedge e_3^* - f_1^*\wedge f_2^*\wedge f_3^* \\
 B_6 &= e_1^*\wedge e_2^*\wedge f_3^* + e_1^*\wedge f_2^*\wedge e_3^* + f_1^*\wedge e_2^*\wedge e_3^* - e_1^*\wedge f_2^*\wedge f_3^* - f_1^*\wedge e_2^*\wedge f_3^* - f_1^*\wedge f_2^*\wedge e_3^*
\end{align*}
The method above shows
\[ H^{0,3}(X_6)/H^3(X_6,\Zbb) \simeq E_6 \]
as complex tori.

We can apply the same steps above to moreover get an exact model for $J(X_6)$ over $\Qbb$, and putting everything together we find that, as above, if the conjecture is true for one of our threefolds, it is not true for any cubic or sextic twists of that threefold.


\begin{theorem} Let $D_1,D_2$ and $D_3$ be squarefree integers, and $E_6(D_j^{k_j})$ the elliptic curve with affine equation $y^2 = x^3 - D_j^{k_j}$, where $k_j\in\{1,2,3\}$ for each $j$. Let $Y_6(D_1^{k_1},D_2^{k_2},D_3^{k_3})$ be a crepant resolution of
\[ (E_6(D_1^{k_1})\times E_6(D_2^{k_2})\times E_6(D_3^{k_3}))/ S, \]
with $S = G_6,H_6,L_6$ or $M_6$. Then
\[ J(Y_6(D_1^{k_1},D_2^{k_2},D_3^{k_3})) = E_6(D_1^{k_1}D_2^{k_2}D_3^{k_3}). \]
There is a model for the intermediate Jacobian such that
\[ L(J(Y_6(D_1^{k_1},D_2^{k_2},D_3^{k_3})),s) = L(s,\chi), \]
where $\chi$ is the Hecke character of $E_6(D_1^{k_1}D_2^{k_2}D_3^{k_3})$, and
\[ L(Y_6(D_1^{k_1},D_2^{k_2},D_3^{k_3}),s) = L(s,\chi^3). \]
if and only if $D_1^{k_1}D_2^{k_2}D_3^{k_3}$ is the cube of an integer. In particular, Conjecture \ref{Yui} is true for $Y_6(D_1^{k_1},D_2^{k_2},D_3^{k_3})$ if and only if it is a cubic twist of $\widetilde{E_6^3/S}$.
\end{theorem}

\begin{proof} If a twist of $\widetilde{E_6^3/S}$ has corresponding $L$-series $L(s,\chi^3\otimes \psi)$, then the intermediate Jacobian has a $\Qbb$-model with $L$-series $L(s,\chi\otimes\psi)$. For the conjecture to be true, we must have
\[ (\chi\otimes\psi)^3 = \chi^3\otimes\psi, \]
and this is not possible if $\psi$ is cubic, or sextic. \end{proof}


Thus, we have verified Yui's conjecture for many examples, but it is not always the case. A somewhat sour irony here is that this construction of rigid threefolds requires the CM automorphisms, and it is precisely this presence of CM that causes the conjecture to fail.

\subsection{Roan's special automorphism}

As $E_3$ and $E_4$ are the only elliptic curves with the CM \textit{automorphisms}, we cannot naively continue in an attempt to construct rigid Calabi-Yau threefolds of CM-type using a triple product of the other CM elliptic curves and a triple product of CM actions. Nevertheless, one may ask if all CM newforms of weight 4 with rational coefficients can be realized as the newform associated to a rigid Calabi-Yau threefold defined over $\Qbb$. If Yui's conjecture is true, this is not possible. E.g., suppose there is a rigid Calabi-Yau threefold of CM-type, defined over $\Qbb$, with associated modular form having level $23^2$ (see \cite{Schuettclassification}). Then the intermediate Jacobian would be an elliptic curve defined over $\Qbb$, with CM by $\Qbb(\sqrt{-23})$, but of course no such curves exist.

If the conjecture is true, we can weaken our hopes to asking if every CM elliptic curve defined over $\Qbb$ is the $\Qbb$-model for the intermediate Jacobian of a rigid Calabi-Yau threefold, of CM-type, defined over $\Qbb$.
Roan provides such an example. Let $E_7$ be the elliptic curve with CM by $\Qbb(\sqrt{-7})$. For a particular model, we can take
\[ E_7 : y^2 + xy = x^3 - x^2 - 2x - 1 \]
which has discriminant 7. Let $\mu = e^{2\pi i/7}$ and
\[ \eta = \mu + \mu^2 + \mu^4 = \frac{-1 + \sqrt{7}i}{2}. \]
While the elliptic curve $E_7$ does not have a CM automorphism, the triple product $E_7^3$ has an automorphism given by
\[ g_\eta := \begin{pmatrix} 0 & 0 & 1 \\ 1 & 0 & \eta + 1 \\ 0 & 1 & \eta \end{pmatrix} \]
such that $E_7^3/\langle g_\eta \rangle$ has a rigid Calabi-Yau threefold $X_7$ for a crepant resolution, and
\[ L(X_7,s) = L(s,\chi_7^3) \]
where $\chi_7$ is the Hecke character of $E_7$. Hence, the unique newform of level 49 with CM corresponds to a rigid Calabi-Yau threefold coming from a triple product of CM elliptic curves. Unfortunately, Roan also shows that this threefold, as well as $\widetilde{E_6^3/M_6}$ are the only such rigid Calabi-Yau threefolds coming from a cyclic construction. The remaining cases with $E_N$, an elliptic curve with CM by $\Qbb(\sqrt{-N})$ for $N = 11,19,43,67,163$, will need more creative (non-cyclic) groups of automorphisms acting on the abelian threefolds if a similar construction is to give a rigid Calabi-Yau threefold.





\subsection{Intermediate Jacobians of higher dimensional Calabi-Yau varieties}

Note that, while the CM is what causes problems for the conjecture on threefolds, if $\psi$ is a biquadratic character, then $\psi^5 = \psi$, and if $\varphi$ is a sextic character, then $\varphi^7 = \varphi$. This inspires an investigation of a generalization of the rigid threefold construction, to higher dimensional Calabi-Yau, in the spirit of \cite{CynkHulek}. One may consider an $n$-fold product
\[ E_j \times E_j \times \cdots \times E_j \]
and $G$, the maximal group of automorphisms given by products of the form
\[ \iota_j^{a_1} \times \iota_j^{a_2} \times \cdots \times \iota_j^{a_n} \]
preserving the holomorphic $n$-form, i.e., such that $\sum a_i \equiv 0\pmod{j}$. When $n$ is odd, only the holomorphic and anti-holomorphic $n$-forms are preserved by the entire group. For any fixed (odd) $n$, we can find all subgroups of $G$ such that a crepant resolution of the quotient by $G$ is a Calabi-Yau $n$-fold $Z_n$ with $h^n(Z_n) = 2$.

We can repeat the computation of the $L$-series, and using the notation above one finds the action of Frobenius on $Z_n$ at good primes $p$, given by matrices
\[ \begin{pmatrix} \alpha_p^n & 0 \\ 0 & \beta_p^n \end{pmatrix} \qquad \text{or} \qquad \begin{pmatrix} 0 & h_p^n \\ k_p^n & 0 \end{pmatrix} \]
when $p$ is a quadratic residue or non-residue respectively, so that
\[ L(Z_n,s) = L(s,\chi^n) \]
where $L(E_j,s) = L(s,\chi)$. Similarly, we may twist each of the underlying elliptic curves to get a birational rigid Calabi-Yau that is a twist of $Z_n$.

The computation of the middle intermediate Jacobian can also be extended. Again, let the $k$-th component in the product $E_\tau^n$ have uniformizing parameter $z_k = x_k + iy_k$, and period $\tau = \alpha + i\beta$, such that the period relation
\[ \begin{pmatrix} e_k^* & f_k^* \end{pmatrix}\begin{pmatrix} 1 & 0 \\ \alpha & \beta \end{pmatrix} = \begin{pmatrix} dx_k & dy_k \end{pmatrix} \]
holds. The holomorphic $n$-form is
\begin{align*}
 \Omega_n &= \bigwedge_{k=1}^n dz_k = \bigwedge_{k=1}^n dx_k + idy_k \\
 &= \bigwedge_{k=1}^n (e_k^* + \alpha f_k^*) + i\beta f_k^* \\
 &= \bigwedge_{k=1}^n e_k^* + \tau f_k^*.
\end{align*}
Hence, we have
\begin{align*}
 \Omega_n &= e_1^*\wedge e_2^* \wedge \cdots \wedge e_n^* + \tau(e_1^*\wedge \cdots \wedge e_{n-1}^* \wedge f_n^* + \cdots + f_1^*\wedge e_2^* \wedge \cdots \wedge e_n^*) \\
&\ +\tau^2(e_1^*\wedge \cdots \wedge e_{n-2}^*\wedge f_{n-1}^*\wedge f_n^* + \cdots + f_1^*\wedge f_2^*\wedge e_3^* \wedge \cdots \wedge e_n^*) + \cdots \\
&\ \cdots + \tau^n(f_1^*\wedge f_2^* \wedge \cdots \wedge f_n^*).
\end{align*}
Our particular choices of $\tau$ are roots of unity of small order, and we have
\[ \zeta^n \in \{1,\zeta,\overline{\zeta}\} \qquad \text{and} \qquad i^n \in \{1,i,-1,-i\} \]
for all $n$, where $\zeta$ is a primitive third root of unity. Hence,
for each of our $n$-folds there are integral classes $A_n,B_n\in H^n(Z_n,\Zbb)$ such that
\[ \begin{pmatrix} A_n & B_n \end{pmatrix}\begin{pmatrix} 1 & 0 \\ \alpha & \beta \end{pmatrix} = \begin{pmatrix} \Real(\Omega_n) & \Img(\Omega_n) \end{pmatrix} \]
and so the intermediate Jacobian $J^{n-1}(Z_n) \simeq E_\tau$ as complex varieties.

\begin{remark} We are interested only in $n$ odd, not only because $h^n(Z_n) > 2$ when $n$ is even, so conjecture \ref{Yui} is not relevant, but because there is no intermediate Jacobian associated to even cohomology, which is the middle cohomology when $n$ is even. \end{remark}

In the arithmetic setting we now have
\[ \Real(\Omega_n) = \lambda^nA_n \]
\[ \Img(\Omega_n) = \lambda^nB_n \]
so that
\[ \begin{pmatrix} \frac{A}{\lambda^n} & \frac{B}{\lambda^n} \end{pmatrix}\begin{pmatrix} \lambda^n & 0 \\ \lambda^n\alpha & \lambda^n\beta \end{pmatrix} = \begin{pmatrix} \Real(\Omega_n) & \Img{\Omega_n} \end{pmatrix} \]
and the $\Qbb$-model satisfies
\[ L(J^{n-1}(Z_n),s) = L(s,\chi) = L(E_\tau,s). \]
If we twist the underlying elliptic curves in $Z_n$ so that the $L$-series of the threefold is $L(s,\chi^n\otimes\psi)$, then the intermediate Jacobian has a $\Qbb$-model with $L$-series $L(s,\chi\otimes\psi)$. Thus, if $n\equiv 1\pmod 4$, the conjecture is true for all rigid threefolds in our construction coming from $E_4$, and if $n\equiv 1\pmod 6$, the conjecture is true for all rigid threefolds in our construction coming from $E_6$.

Let $n$ be an odd positive integer. The natural question is then as follows. Let $Z$ is a rigid Calabi-Yau $n$-fold of CM-type, defined over a number field $F$, having intermediate Jacobian $J^{n-1}(Z)$ with CM by a number field $K = \Qbb(\sqrt{-D})$, where the CM automorphism of $J^{n-1}(Z)$ has order $m$. Then, if $n \equiv 1\pmod{m}$ and
\[ L(J^{n-1}(Z),s) = \left\{ \begin{array}{ll} L(\chi,s)L(\overline{\chi},s) & \text{if}\ K\subset F, \\
 L(\chi,s) & \text{otherwise} \end{array}\right. \]
then must we have
\[ L(Z,s) = \left\{ \begin{array}{ll} L(\chi^n,s)L(\overline{\chi}^n,s) & \text{if}\ K\subset F, \\
 L(\chi^n,s) & \text{otherwise?} \end{array} \right. \]
Otherwise, if $n\not\equiv 1\pmod{m}$, does the above only fail if we exploit the CM property?

\section{Remarks on special values of $L$-functions}

With a $\Qbb$-model of the intermediate Jacobians associated to any $\Qbb$-model of our rigid Calabi-Yau threefolds via the conjecture, we are able to define the $L$-functions of the respective varieties and investigate their behaviour in their critical strips. 

If $\chi$ is a Hecke character of conductor $\ffrak$, there is a primitive character $\chi_\ffrak$ equivalent to $\chi$ in the sense that
\[ L(s,\chi) = L(s,\chi_\ffrak). \]
In this section, we will always assume we are dealing with a primitive Hecke character. If $K$ is imaginary quadratic with discriminant $D = |D|$ and $\chi$ a primitive Hecke character of $K$ of infinity type $1$, then the completed $L$-function
\[ \Lambda(s,\chi) = (D\Nb\ffrak(\chi))^{s/2}\cdot 2(2\pi)^{-s}\Gamma(s)L(s,\chi) \]
satisfies the functional equation
\[ \Lambda(s,\chi) = W(\chi)\Lambda(2-s,\overline{\chi}), \]
where $W(\chi) = \pm 1$ is the \textit{root number}, and $\Gamma(s)$ is the complex gamma function that analytically continues the factorial. Moreover, this extends to powers of $\chi$. For an integer $n$, let $\chi_n$ denote the primitive character associated to $\chi^n$. 
The completed $L$-function
\[ \Lambda(s,\chi_n) = \left\{\begin{array}{ll} (D\Nb\ffrak(\chi))^{s/2}\cdot 2(2\pi)^{-s}\Gamma(s)L(s,\chi_n) & \text{if}\ n\ \text{is odd}, \\
 D^{s/2}\cdot 2(2\pi)^{-s}\Gamma(s)L(s,\chi_n) & \text{otherwise} \end{array} \right. \]
satisfies the functional equation
\[ \Lambda(s,\chi_n) = W(\chi_n)\Lambda(w+1-s,\chi_w) \]
where
\[ W(\chi_n) = \left\{\begin{array}{ll} (-1)^{(n-1)/2}W(\chi) & \text{if}\ n\ \text{is odd}, \\ 1 & \text{otherwise} \end{array} \right. . \]
This follows by use of the adelic language for Hecke characters and considering local factors as in Tate's thesis. For full details, see Rohrlich's work in \cite{Rohrlich}.

As the completed $L$-functions are simply the analytic continuation of the $L(s,\chi_n)$ we abuse notation when computing with them, and are interested in computing the central values in the critical strip, namely
\[ L(n/2,\chi_n). \]
These can be very difficult to compute exactly, but a method of Waldspurger relates these values to certain half-integral weight modular forms. In this direction, one can show that the critical values $L(X_4(-D),3/2)$ are determined by certain cusp forms
\begin{align*}
 f_1 &= q - 3q^9 - 4q^{17} + 25q^{25} - 4q^{33} - 48q^{41} + q^{49} + 20q^{57} + 48q^{65} -
    4q^{73} - 27q^{81} + 68q^{89} - 76q^{97} + O(q^{105}) \\
 f_2 &= -q^3 + 5q^{11} - 7q^{19} + 2q^{35} + q^{43} + 14q^{51} - 13q^{59} + q^{67} - 27q^{75} +
    7q^{83} + 26q^{91} + 15q^{99} + O(q^{107})
\end{align*}
of weight $5/2$ and level $128$. 

These methods are very different from those in this work, and somewhat involved, so we leave the details to future work in \cite{Molnar2}, but as an example of what is possible, we find results of the form
\begin{theorem} Let $E_4 : y^2 = x^3 - x$, and $X_4(-D)$ a crepant resolution of $E_4(-D)^3/G_4$. For any odd square-free $m\in\Nbb$ we have
\[ L(X_4(-D),2) = \left\{\begin{array}{ll} \frac{a_D^2}{\alpha\sqrt{D^3}} & \text{if}\ D\equiv 1\pmod 8, \\
\frac{b_D^2}{\beta\sqrt{D^3}} & \text{if}\ D\equiv 3\pmod 8, \\
0 & \text{if}\ D\equiv 5,7\pmod 8. \end{array}\right. \]
where $\alpha,\beta\in\Cbb^\times$ are related to the real periods of $X_4(-D)$, and
\[ f_1 = \sum_{D=1}^\infty a_Dq^D, \]
\[ f_2 = \sum_{D=1}^\infty b_Dq^D. \] \end{theorem}
%
%
In particular, we see which twists of the Calabi-Yau threefolds have vanishing central critical values and which do not.

\bibliographystyle{amsplain}
\bibliography{refs}

\end{document}